\documentclass[reqno]{amsart}
\usepackage{amssymb}
\usepackage{amsfonts}

\newtheorem{theorem}{Theorem}
\theoremstyle{plain}

\newtheorem{lemma}{Lemma}

\numberwithin{equation}{section}
\numberwithin{theorem}{section}
\numberwithin{algorithm}{section}
\numberwithin{axiom}{section}
\numberwithin{case}{section}
\numberwithin{claim}{section}
\numberwithin{conclusion}{section}
\numberwithin{condition}{section}
\numberwithin{conjecture}{section}
\numberwithin{corollary}{section}
\numberwithin{criterion}{section}
\numberwithin{definition}{section}
\numberwithin{example}{section}
\numberwithin{exercise}{section}
\numberwithin{lemma}{section}
\numberwithin{notation}{section}
\numberwithin{problem}{section}
\numberwithin{proposition}{section}
\numberwithin{remark}{section}
\numberwithin{solution}{section}

\input{tcilatex}

\begin{document}
\title[A remark on the fourth order Q curvature]{A remark on the fourth
order Q curvature on manifolds with dimension at least 5}
\author{Fengbo Hang}
\address{Courant Institute, New York University, 251 Mercer Street, New York
NY 10012}
\email{fengbo@cims.nyu.edu}

\begin{abstract}
We discuss the dichotomy for fourth order Q curvature on Riemannian
manifolds with positive Yamabe invariant and dimension at least 5.
\end{abstract}

\maketitle

Let $\left( M^{n},g\right) $ be a smooth compact Riemannian manifold with
dimension $n\geq 5$. The conformal Laplacian is given by%
\begin{equation}
L=-\frac{4\left( n-1\right) }{n-2}\Delta +R.  \label{eq0.1}
\end{equation}%
The associated Yamabe invariant%
\begin{equation}
Y\left( g\right) =\inf_{u\in H^{1}\left( M\right) \backslash \left\{
0\right\} }\frac{\int_{M}\left( \frac{4\left( n-1\right) }{n-2}\left\vert
\nabla u\right\vert ^{2}+Ru^{2}\right) d\mu }{\left\Vert u\right\Vert _{L^{%
\frac{2n}{n-2}}}^{2}}.  \label{eq0.2}
\end{equation}%
The fourth order Q curvature is given by%
\begin{equation}
Q=-\frac{1}{2\left( n-1\right) }\Delta R-\frac{2}{\left( n-2\right) ^{2}}%
\left\vert Rc\right\vert ^{2}+\frac{n^{3}-4n^{2}+16n-16}{8\left( n-1\right)
^{2}\left( n-2\right) ^{2}}R^{2}.  \label{eq0.3}
\end{equation}%
The Paneitz operator is defined as%
\begin{eqnarray}
&&P\varphi  \label{eq0.4} \\
&=&\Delta ^{2}\varphi +\frac{4}{n-2}\func{div}\left( Rc\left( \nabla \varphi
,e_{i}\right) e_{i}\right) -\frac{n^{2}-4n+8}{2\left( n-1\right) \left(
n-2\right) }\func{div}\left( R\nabla \varphi \right) +\frac{n-4}{2}Q\varphi 
\notag
\end{eqnarray}%
Here $e_{1},\cdots ,e_{n}$ is a local orthonormal frame with respect to $g$.
For background of the Paneitz operator and Q curvature one can see \cite{HY2}
and the references therein. For $u\in H^{2}\left( M\right) $, the quadratic
form associated with $P$ is given by%
\begin{eqnarray}
&&E\left( u\right)  \label{eq0.5} \\
&=&\int_{M}\left[ \left( \Delta u\right) ^{2}-\frac{4}{n-2}Rc\left( \nabla
u,\nabla u\right) +\frac{n^{2}-4n+8}{2\left( n-1\right) \left( n-2\right) }%
R\left\vert \nabla u\right\vert ^{2}+\frac{n-4}{2}Qu^{2}\right] d\mu . 
\notag
\end{eqnarray}%
The analogy of $Y\left( g\right) $ is given by%
\begin{equation}
Y_{4}\left( g\right) =\inf_{u\in H^{2}\left( M\right) \backslash \left\{
0\right\} }\frac{E\left( u\right) }{\left\Vert u\right\Vert _{L^{\frac{2n}{%
n-4}}}^{2}}.  \label{eq0.6}
\end{equation}%
Note that $Y_{4}\left( g\right) $ is a finite number with the same sign as
the first eigenvalue of Paneitz operator. Let $\left[ g\right] $ be the
conformal class of $g$.

\begin{theorem}
\label{thm0.1}Let $\left( M^{n},g\right) $ be a smooth compact Riemannian
manifold with dimension $n\geq 5$ and $Y\left( g\right) >0$. Then we have

\begin{enumerate}
\item $Y_{4}\left( g\right) >0\Longleftrightarrow \exists \widetilde{g}\in %
\left[ g\right] $ s.t. $\widetilde{R}>0$ and $\widetilde{Q}>0.$

\item $Y_{4}\left( g\right) =0\Longleftrightarrow \exists \widetilde{g}\in %
\left[ g\right] $ s.t. $\widetilde{R}>0$ and $\widetilde{Q}=0.$

\item $Y_{4}\left( g\right) <0\Longleftrightarrow \exists \widetilde{g}\in %
\left[ g\right] $ s.t. $\widetilde{R}>0$ and $\widetilde{Q}<0.$
\end{enumerate}
\end{theorem}

Note that the similar dichotomy for scalar curvature was given in \cite{KW,
LP}. The direction "$\Rightarrow $" of case (1) was proved in \cite[Theorem
1.1]{GHL} for $n\geq 6$ and more recently in \cite[Theorem 1.1]{Li} for $n=5$%
. The direction "$\Leftarrow $" of case (1) was proved in \cite{XY} for $%
n\geq 6$ and in \cite[Proposition 2.3]{GM} for $n=5$. The direction "$%
\Leftarrow $" of case (2) follows from the arguments in \cite{GM, XY}.
Before we proceed, we want to point out that Q curvature does not behave
well in general for conformal classes with negative Yamabe invariant. In
fact, it is shown in \cite{Z} for $n\geq 5$, there exists $\left(
M^{n},g\right) $, a smooth compact Riemannian manifold with $Y\left(
g\right) <0$, for which one can not find any $\widetilde{g}\in \left[ g%
\right] $ such that $\widetilde{Q}$ is strictly positive or strictly
negative or identically zero.

We will use the same notations as \cite[Section 3]{HY2}. Under the
assumption $Y\left( g\right) >0$, let%
\begin{eqnarray}
&&H\left( p,q\right)  \label{eq0.7} \\
&=&2^{\frac{n-6}{n-2}}n^{-\frac{2}{n-2}}\left( n-1\right) ^{\frac{n-4}{n-2}%
}\left( n-2\right) ^{-1}\left( n-4\right) ^{-1}\omega _{n}^{-\frac{2}{n-2}%
}G_{L}\left( p,q\right) ^{\frac{n-4}{n-2}},  \notag
\end{eqnarray}%
and%
\begin{eqnarray}
&&\Gamma _{1}\left( p,q\right)  \label{eq0.8} \\
&=&2^{\frac{n-6}{n-2}}n^{-\frac{2}{n-2}}\left( n-1\right) ^{\frac{n-4}{n-2}%
}\left( n-2\right) ^{-3}\omega _{n}^{-\frac{2}{n-2}}G_{L}\left( p,q\right) ^{%
\frac{n-4}{n-2}}\left\vert Rc_{G_{L,p}^{\frac{4}{n-2}}g}\right\vert
_{g}^{2}\left( q\right) .  \notag
\end{eqnarray}%
Here $\omega _{n}$ is the volume of unit ball in $\mathbb{R}^{n}$ and $%
G_{L}\left( p,q\right) $ is the Green's function of the conformal Laplacian
operator $L$. Note that by the calculation in \cite[Section 2]{HY1},%
\begin{equation}
\Gamma _{1}\left( p,q\right) =O\left( \overline{pq}^{4-n}\right) ,
\label{eq0.9}
\end{equation}%
(similar estimates are also true for higher order derivatives of $\Gamma
_{1} $) here $\overline{pq}$ denotes the distance between $p$ and $q$ and%
\begin{equation}
P_{q}H\left( p,q\right) =\delta _{p}\left( q\right) -\Gamma _{1}\left(
p,q\right) .  \label{eq0.10}
\end{equation}%
It follows that (see \cite[Proposition 4.1]{Li}) for any $f\in C^{\infty
}\left( M\right) $,%
\begin{equation}
PT_{H}f=f-T_{\Gamma _{1}^{T}}f.  \label{eq0.11}
\end{equation}%
Here 
\begin{equation}
\Gamma _{1}^{T}\left( p,q\right) =\Gamma _{1}\left( q,p\right)
\label{eq0.12}
\end{equation}%
and%
\begin{equation}
T_{H}f\left( p\right) =\int_{M}H\left( p,q\right) f\left( q\right) d\mu
\left( q\right) .  \label{eq0.13}
\end{equation}%
Same notations apply for $T_{\Gamma _{1}^{T}}$. In general $\Gamma
_{1}^{T}\neq \Gamma _{1}$. However $T_{\Gamma _{1}^{T}}$ is the transpose of 
$T_{\Gamma _{1}}$, in particular they have the same spectrum and hence the
same spectral radius i.e. $r_{\sigma }\left( T_{\Gamma _{1}^{T}}\right)
=r_{\sigma }\left( T_{\Gamma _{1}}\right) $. As point out in \cite[Section 3]%
{HY2}, $r_{\sigma }\left( T_{\Gamma _{1}}\right) $ is a conformal invariant.
Note that if $\left( M,g\right) $ is not conformal diffeomorphic to the
standard sphere, then for any $p$, $\Gamma _{1}\left( p,\cdot \right) $ is
not identically zero (see \cite[Section 2]{HY1}), hence%
\begin{equation*}
r_{\sigma }\left( T_{\Gamma _{1}}\right) \geq \min_{p\in M}\int_{M}\Gamma
_{1}\left( p,q\right) d\mu \left( q\right) >0.
\end{equation*}%
The key observation in \cite{Li} is the following

\begin{lemma}
\label{lem0.1} (\cite[Lemma 3.1]{Li}) Let $f\in C^{\infty }\left( M\right) $%
, $f\geq 0$ and not identically zero, then%
\begin{equation}
R_{\left( T_{H}f\right) ^{\frac{4}{n-4}}g}>0.  \label{eq0.14}
\end{equation}
\end{lemma}

Note that since $H\left( p,q\right) >0$, $T_{H}f>0$. The statement above is
slightly different from \cite[Lemma 3.1]{Li} that $R_{\left( T_{H}f\right) ^{%
\frac{4}{n-4}}g}$ is always strictly positive even if $f$ is not strictly
positive everywhere. This technical change will be useful later. To see this
let us recall the calculation in the proof of \cite[Lemma 3.1]{Li}: we only
need to show%
\begin{eqnarray}
&&-\frac{4\left( n-1\right) }{n-2}\Delta \left( T_{H}f\right) ^{\frac{n-2}{%
n-4}}+R\left( T_{H}f\right) ^{\frac{n-2}{n-4}}  \label{eq0.15} \\
&=&\left( T_{H}f\right) ^{\frac{2}{n-4}}\left[ -\frac{4\left( n-1\right) }{%
n-4}\Delta T_{H}f-\frac{8\left( n-1\right) }{\left( n-4\right) ^{2}}\frac{%
\left\vert \nabla T_{H}f\right\vert ^{2}}{T_{H}f}+RT_{H}f\right]  \notag
\end{eqnarray}%
is strictly positive. Using%
\begin{equation*}
-\frac{4\left( n-1\right) }{n-2}\Delta _{p}H\left( p,q\right) ^{\frac{n-2}{%
n-4}}+R\left( p\right) H\left( p,q\right) ^{\frac{n-2}{n-4}}=0
\end{equation*}%
for $p\neq q$, one has%
\begin{equation}
-\frac{4\left( n-1\right) }{n-4}\Delta _{p}H\left( p,q\right) -\frac{8\left(
n-1\right) }{\left( n-4\right) ^{2}}\frac{\left\vert \nabla _{p}H\left(
p,q\right) \right\vert ^{2}}{H\left( p,q\right) }+R\left( p\right) H\left(
p,q\right) =0.  \label{eq0.16}
\end{equation}%
Plug this identity into (\ref{eq0.15}) and as in \cite[Lemma 3.1]{Li} we
arrive at 
\begin{eqnarray}
&&-\frac{4\left( n-1\right) }{n-4}\Delta T_{H}f\left( p\right) -\frac{%
8\left( n-1\right) }{\left( n-4\right) ^{2}}\frac{\left\vert \nabla
T_{H}f\left( p\right) \right\vert ^{2}}{T_{H}f\left( p\right) }+R\left(
p\right) T_{H}f\left( p\right)  \label{eq0.17} \\
&=&\frac{4\left( n-1\right) }{\left( n-4\right) ^{2}}\frac{1}{T_{H}f\left(
p\right) }\int_{M\times M}\left\vert \nabla _{p}\log H\left( p,q_{1}\right)
-\nabla _{p}\log H\left( p,q_{2}\right) \right\vert ^{2}\cdot  \notag \\
&&H\left( p,q_{1}\right) H\left( p,q_{2}\right) f\left( q_{1}\right) f\left(
q_{2}\right) d\mu \left( q_{1}\right) d\mu \left( q_{2}\right) .  \notag
\end{eqnarray}%
Note that by unique continuation property, fix $p$, for any $q_{0}\in M$,
any $U\left( q_{0}\right) $, an open neighborhood of $q_{0}$, there exists $%
q_{1},q_{2}\in U\left( q_{0}\right) \backslash \left\{ p\right\} $ such that%
\begin{eqnarray*}
&&\left\vert \nabla _{p}\log H\left( p,q_{1}\right) -\nabla _{p}\log H\left(
p,q_{2}\right) \right\vert ^{2} \\
&=&\left( \frac{n-4}{n-2}\right) ^{2}\left\vert \nabla _{p}\log G_{L}\left(
p,q_{1}\right) -\nabla _{p}\log G_{L}\left( p,q_{2}\right) \right\vert ^{2}
\\
&>&0
\end{eqnarray*}%
In fact, if this is not the case, then $\nabla _{p}\log G_{L}\left(
p,q\right) $ is independent of $q\in U\left( q_{0}\right) \backslash \left\{
p\right\} $. Let $\xi \in M_{p}$ such that%
\begin{equation*}
\nabla _{p}\log G_{L}\left( p,q\right) =\xi
\end{equation*}%
for $q\in U\left( q_{0}\right) \backslash \left\{ p\right\} $. Then%
\begin{equation*}
\nabla _{p}G_{L}\left( p,q\right) -G_{L}\left( p,q\right) \xi =0\text{ for }%
q\in U\left( q_{0}\right) \backslash \left\{ p\right\} .
\end{equation*}%
Since%
\begin{equation*}
L_{q}\left( \nabla _{p}G_{L}\left( p,q\right) -G_{L}\left( p,q\right) \xi
\right) =0\text{ for }q\in M\backslash \left\{ p\right\} ,
\end{equation*}%
by unique continuation property we see%
\begin{equation*}
\nabla _{p}G_{L}\left( p,q\right) -G_{L}\left( p,q\right) \xi =0\text{ for }%
q\in M\backslash \left\{ p\right\} ,
\end{equation*}%
i.e 
\begin{equation*}
\nabla _{p}\log G_{L}\left( p,q\right) =\xi \text{ for }q\in M\backslash
\left\{ p\right\} .
\end{equation*}%
This is impossible since $\left\vert \nabla _{p}\log G_{L}\left( p,q\right)
\right\vert \rightarrow \infty $ as $q\rightarrow p$. A contradiction.

Since $f$ is not identically zero, we assume $f\left( q_{0}\right) >0$.
Hence on a small neighborhood $U\left( q_{0}\right) $, $f>\frac{1}{2}f\left(
q_{0}\right) >0$. It follows that the integral in (\ref{eq0.17}) is strictly
positive. This finishes the proof of Lemma \ref{lem0.1}.

To continue we can assume $r_{\sigma }\left( T_{\Gamma _{1}}\right) >0$
(otherwise $\left( M,g\right) $ is conformal diffeomorphic to the standard
sphere and Theorem \ref{thm0.1} is clear). The classical Krein-Rutman
theorem tells us we can find $f\in C\left( M\right) $, $f\geq 0$, $f$ not
identically zero and $T_{\Gamma _{1}^{T}}f=r_{\sigma }\left( T_{\Gamma
_{1}}\right) f$. Standard bootstrap argument shows $f\in C^{\infty }\left(
M\right) $. By (\ref{eq0.11}) we have%
\begin{equation*}
PT_{H}f=\left( 1-r_{\sigma }\left( T_{\Gamma _{1}}\right) \right) f.
\end{equation*}

If $r_{\sigma }\left( T_{\Gamma _{1}}\right) =1$, then $PT_{H}f=0$. Let $%
\widetilde{g}=\left( T_{H}f\right) ^{\frac{4}{n-4}}g$, then $\widetilde{R}>0$
and $\widetilde{Q}=0$. It follows from the proof in \cite{GM, XY} that $%
P\geq 0$ and $\ker P=\left\{ \text{constant functions}\right\} $. Hence $%
Y_{4}\left( g\right) =0$.

If $r_{\sigma }\left( T_{\Gamma _{1}}\right) >1$, then $PT_{H}f\leq 0$ and
not identically, hence $R_{\left( T_{H}f\right) ^{\frac{4}{n-4}}g}>0$ and $%
Q_{\left( T_{H}f\right) ^{\frac{4}{n-4}}g}\leq 0$ and not identically zero.
We claim it follows that $\exists \widetilde{g}\in \left[ g\right] $ such
that $\widetilde{R}>0$ and $\widetilde{Q}<0$. Hence $Y_{4}\left( g\right) <0$%
. The claim follows from the Lemma \ref{lem0.2} below (using $\left(
T_{H}f\right) ^{\frac{4}{n-4}}g$ as background metric).

\begin{lemma}
\label{lem0.2}Let $\left( M^{n},g\right) $ be a smooth compact Riemannian
manifold with dimension $n\geq 5$. If $R>0$, $Q\leq 0$ and not identically
zero, then $\exists \widetilde{g}\in \left[ g\right] $ such that $\widetilde{%
R}>0$ and $\widetilde{Q}<0$.
\end{lemma}

\begin{proof}
Note that $P1=\frac{n-4}{2}Q\leq 0$ and not identically $0$. Let%
\begin{equation*}
Z=\left\{ p\in M:Q\left( p\right) =0\right\} .
\end{equation*}%
Then $Z$ is a closed subset in $M$. We can assume $Z$ is nonempty. Let $%
p_{0}\in M$ such that $Q\left( p_{0}\right) <0$. We define%
\begin{equation*}
V=\left\{ p\in M:Q\left( p\right) <\frac{1}{2}Q\left( p_{0}\right) \right\} .
\end{equation*}%
Then $V$ is open nonempty and $Z\cap \overline{V}=\emptyset $. Let $\chi \in
C_{c}^{\infty }\left( M\backslash \overline{V}\right) $ such that $\chi \geq
0$, $\left. \chi \right\vert _{Z}=1$. Then we can find $\varphi \in
C_{c}^{\infty }\left( V\right) $ such that for any $\psi \in \ker P$,%
\begin{equation*}
\int_{M}\left( \chi +\varphi \right) \psi d\mu =0.
\end{equation*}%
Hence there exists $v\in C^{\infty }\left( M\right) $ such that $Pv=\chi
+\varphi $. Then for $\varepsilon >0$ small enough, we have%
\begin{equation*}
P\left( 1-\varepsilon v\right) =\frac{n-4}{2}Q-\varepsilon \chi -\varepsilon
\varphi <0.
\end{equation*}%
We let $\widetilde{g}=\left( 1-\varepsilon v\right) ^{\frac{4}{n-4}}g$, then
for $\varepsilon >0$ small enough, $\widetilde{R}>0$ and $\widetilde{Q}<0$.

Here is one way to construct $\varphi $. Fix $\eta \in C_{c}^{\infty }\left(
V\right) $ such that $\eta \left( p_{0}\right) =1$. Let $\psi _{1},\cdots
,\psi _{m}$ be a base for $\ker P$, then $\eta \psi _{1},\cdots ,\eta \psi
_{m}$ are linearly independent (indeed if $t_{1},\cdots ,t_{m}\in \mathbb{R}$
s.t. $t_{1}\eta \psi _{1}+\cdots +t_{m}\eta \psi _{m}=0$, then $t_{1}\psi
_{1}+\cdots +t_{m}\psi _{m}=0$ on $\left\{ \eta >0\right\} $. Since $P\left(
t_{1}\psi _{1}+\cdots +t_{m}\psi _{m}\right) =0$, it follows from unique
continuation property that $t_{1}\psi _{1}+\cdots +t_{m}\psi _{m}=0$ on $M$.
So $t_{1},\cdots ,t_{m}$ are all equal to $0$). The Gram matrix $\left[
\int_{M}\eta ^{2}\psi _{i}\psi _{j}d\mu \right] _{1\leq i,j\leq m}$ must be
positive definite, hence%
\begin{equation}
\det \left[ \int_{M}\eta ^{2}\psi _{i}\psi _{j}d\mu \right] >0.
\label{eq0.18}
\end{equation}%
We choose $\varphi =c_{1}\eta ^{2}\psi _{1}+\cdots +c_{m}\eta ^{2}\psi _{m}$%
, then by (\ref{eq0.18}) we can find unique $c_{1},\cdots ,c_{m}\in \mathbb{R%
}$ such that%
\begin{equation*}
\int_{M}\left( \chi +\varphi \right) \psi _{j}d\mu =0
\end{equation*}%
for $1\leq j\leq m$. This gives us the needed $\varphi $.
\end{proof}

Same argument gives us

\begin{lemma}
\label{lem0.3}Let $\left( M^{n},g\right) $ be a smooth compact Riemannian
manifold with dimension $n\geq 5$. If $R>0$, $Q\geq 0$ and not identically
zero, then $\exists \widetilde{g}\in \left[ g\right] $ such that $\widetilde{%
R}>0$ and $\widetilde{Q}>0$.
\end{lemma}

Lemma \ref{lem0.3} also follows from the existence of constant Q curvature
theorem in \cite{GM, HY3}. But the argument here is more elementary.

If $r_{\sigma }\left( T_{\Gamma _{1}}\right) <1$, then $PT_{H}f\geq 0$ and
not identically, hence $R_{\left( T_{H}f\right) ^{\frac{4}{n-4}}g}>0$ and $%
Q_{\left( T_{H}f\right) ^{\frac{4}{n-4}}g}\geq 0$ and not identically zero.
Using Lemma \ref{lem0.3} we see $\exists \widetilde{g}\in \left[ g\right] $
such that $\widetilde{R}>0$ and $\widetilde{Q}>0$. It follows from \cite{GM,
XY} that $P>0$, hence $Y_{4}\left( g\right) >0$.

Let us go back to Theorem \ref{thm0.1}. If $Y_{4}\left( g\right) >0$, we
must have $r_{\sigma }\left( T_{\Gamma _{1}}\right) <1$, hence $\exists 
\widetilde{g}\in \left[ g\right] $ such that $\widetilde{R}>0$ and $%
\widetilde{Q}>0$. If $Y_{4}\left( g\right) =0$, we must have $r_{\sigma
}\left( T_{\Gamma _{1}}\right) =1$, hence $\exists \widetilde{g}\in \left[ g%
\right] $ such that $\widetilde{R}>0$ and $\widetilde{Q}=0$. If $Y_{4}\left(
g\right) <0$, we must have $r_{\sigma }\left( T_{\Gamma _{1}}\right) >1$,
hence$\exists \widetilde{g}\in \left[ g\right] $ such that $\widetilde{R}>0$
and $\widetilde{Q}<0$. The discussion also shows we could use $r_{\sigma
}\left( T_{\Gamma _{1}}\right) $ to produce the dichotomy as well.

\end{document}